\documentclass[11pt,a4paper]{article}

\usepackage{verbatim}
\usepackage{iftex}
\ifPDFTeX
  \usepackage[T1]{fontenc}
  \usepackage{lmodern}
\fi
\usepackage[english]{babel}
\usepackage{amsmath,amssymb,amsthm,mathtools}
\numberwithin{equation}{section}
\usepackage{bm}
\usepackage{microtype}
\usepackage{xcolor}
\usepackage{float}
\usepackage{enumerate,enumitem}
\usepackage[margin=2.3cm]{geometry}
\usepackage{tikz}
\usetikzlibrary{arrows.meta,calc,fit,positioning,shapes.geometric}
\usepackage{todonotes}
\usepackage[colorlinks,linkcolor=blue,citecolor=blue,urlcolor=blue]{hyperref}
\usepackage[capitalize,nameinlink,noabbrev]{cleveref}
\counterwithin{figure}{section}

\newtheorem{theorem}{Theorem}[section]
\newtheorem*{thm-non}{Theorem}
\newtheorem{conjecture}[theorem]
{Conjecture}

\newtheorem{proposition}[theorem]{Proposition}
\newtheorem{lemma}[theorem]{Lemma}

\newtheorem{claim}[theorem]{Claim}

\newtheorem{problem}[theorem]{Problem}
\theoremstyle{definition}
\newtheorem{definition}[theorem]{Definition}
\newtheorem*{defn-non}{Definition}

\newcounter{propcounter}

\newenvironment{poc}
{
\begin{proof}[Proof of claim]}
  {
\end{proof}}

\newcommand{\eps}{\varepsilon}
\newcommand{\circum}{\operatorname{circ}}

\title{\LARGE Small circumference in regular sublinear expanders}

\author{
Yaobin Chen \quad Hong Liu \quad Xin Wei \quad Fan Yang
\thanks{All authors are supported by the Institute for Basic Science
(IBS-R029-C4). Fan Yang is also supported by Natural Science Foundation
of China (12301447), Natural Science Foundation of Shandong Province
(ZR2024QA056), and China Scholarship Council. Emails: \texttt{\{ybchen,hongliu,weixinma\}@ibs.re.kr},
\texttt{fyang@sdu.edu.cn}}\\[0.6em]
\small Extremal Combinatorics and Probability Group (ECOPRO),\\[-0.1em]
\small Institute for Basic Science (IBS), Daejeon, South Korea\\[0.4em]
}
\date{}

\begin{document}

\maketitle

\begin{abstract}
Sublinear expansion is weak enough to be extracted from arbitrary graphs while retaining nearly all of their average degree, yet it has proved strong enough to force global structures in many sparse extremal problems. Letzter, Methuku and Sudakov [JLMS 2026] proved the existence of nearly Hamilton cycles in sufficiently dense regular sublinear expanders. Montgomery [ICM 2026] subsequently conjectured that, every $d$-regular sublinear expander with $d$ sufficiently large (but constant) is Hamiltonian. We disprove this conjecture in a strong form by constructing $n$-vertex $d$-regular sublinear expanders with  $d=\left(\frac12+o(1)\right)\log^2 n$, which can forbid any cycle covering an arbitrarily small given positive constant portion of its vertices.
%\red{for every fixed \(\delta>0\), we construct \(n\)-vertex \(d\)-regular sublinear expanders with \(d=\left(\frac12+o(1)\right)\log^2 n\) in which every cycle has length at most \(\delta n\), with the expansion parameters depending on \(\delta\).} 
The construction blows up one side of a biregular Ramanujan graph into almost-complete blocks and keeps the other side as a sparse vertex separator.

We also prove a similar statement for a closely related notion of edge expanders. For every sufficiently small $\gamma>0$, there is an infinite family of $n$-vertex $d$-regular $\gamma$-edge-expanders with $d=\Theta(\gamma^{-1})$ and circumference
$O(\gamma n)$, matching the standard lower bound
$\Omega(\gamma n)$. The construction also comes from an expanding regular core such that each vertex has an almost-complete graph attached to it. The core guarantees edge expansion, while the single-vertex attachments confine every cycle.
\end{abstract}

\section{Introduction}
Hamiltonicity is a canonical test of whether local density or pseudorandomness forces a genuinely spanning structure. 
Classical criteria such as those of Dirac \cite{Dirac} and Ore \cite{Ore} impose linear degree conditions, while Chv\'atal--Erd\H{o}s \cite{ChvatalErdos} used connectivity together with a bound on the independence number. 
In sparse graphs, degree alone is insufficient, and expansion has become the principal replacement: it
supplies the global connectivity needed to join local pieces into a
spanning cycle. This philosophy is particularly successful for
spectral expanders. For instance, a conjecture of Krivelevich and
Sudakov \cite{Sudakov1}, asserting that a sufficiently large constant
ratio between the degree and the nontrivial eigenvalues forces
Hamiltonicity, was recently resolved by Dragani\'c, Montgomery,
Munh\'a Correia, Pokrovskiy and Sudakov \cite{Sudakov2}.

Linear expansion, however, is too strong to be recovered from an
arbitrary sparse graph without a substantial loss of density.
Sublinear expanders, introduced by Koml\'os and Szemer\'edi
\cite{KS1,KS2}, were designed precisely to bridge this gap: their
expansion rate deteriorates only polylogarithmically with the size of
the set, and every graph contains such an expander with nearly the
same average degree. This extraction principle has become a standard
first step in sparse extremal graph theory. It underlies major
progress on subdivisions, minors, immersions, cycles of prescribed lengths, cycle decompositions and cycles with many chords; see
\cite{3Mont,4chords,Let-survey,16Liu-JAMS,immersion,minor2,MontgomerySurvey,minor1,Yang2}.
In particular, structural results for sublinear expanders often
transfer to much broader classes of sparse graphs. Recent trends have focused on the existence of Hamilton cycles in expanders with restrictions.

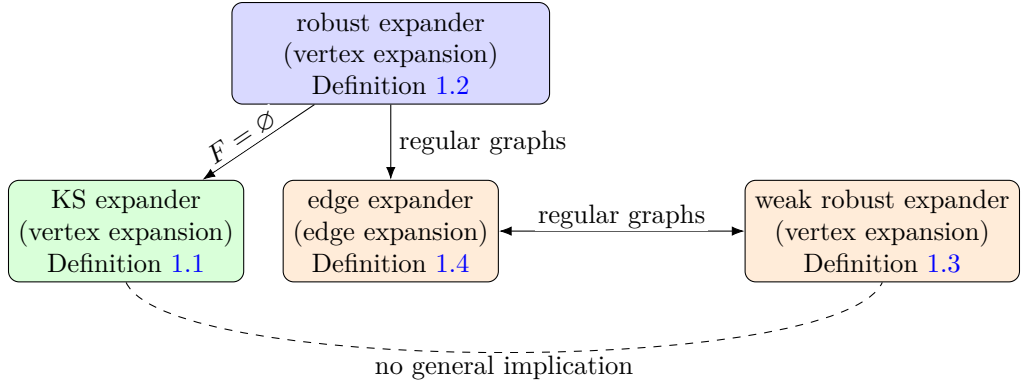
\begin{figure}[H]
\centering
\begin{tikzpicture}[
  box/.style={draw,rounded corners,align=center,minimum width=3.25cm,
              minimum height=0.85cm},
  every node/.style={font=\small},
  >=Latex]
  \node[box,fill=blue!15,minimum width=4.2cm] (hkl) at (0.5,2.35)
    {robust expander\\(vertex expansion)\\Definition \ref{defn:HKL}};
  \node[box,fill=green!15,minimum width=2.7cm] (ks) at (-3.0,0)
    {KS expander \\(vertex expansion)\\Definition \ref{def:KS}};
  \node[box,fill=orange!15,minimum width=2.7cm] (bj) at (0.5,0)
    {edge expander\\(edge expansion)\\Definition \ref{BJ-ga}};
  \node[box,fill=orange!15,minimum width=3.4cm] (lms) at (7,0)
    {weak robust expander\\(vertex expansion)\\Definition \ref{defweakro}};
  \draw[->] (hkl) -- node[pos=.58,above,sloped,fill=white,
    inner sep=1pt] {$F=\varnothing$} (ks);
  \draw[->] (hkl) --
    node[midway,right=2pt,fill=white,inner sep=1pt]
    {regular graphs} (bj);
  \draw[<->] (bj) --
    node[midway,above,fill=white,inner sep=1pt]
    {regular graphs} (lms);
  \draw[dashed] (ks.south) to[out=-65,in=-115,looseness=.35]
    node[below,fill=white,inner sep=1pt]
    {no general implication} (lms.south);
\end{tikzpicture}
\caption{For regular graphs, the robust expansion implies both KS expansion and edge expansion, whereas KS expanders and edge expanders are incomparable.}
\label{fig:relations}
\end{figure}

There are several variants of sublinear expansion in the literature,
and the distinction between them matters for spanning structures.  We recall below four commonly-used notions. We write $d(G)$ for the average degree of $G$. For $X\subseteq V(G)$, let $\overline{X}= V(G)\setminus X$ and let 
$N_G(X):=\big(\cup_{v\in X}N_G(v)\big)\setminus X$
be the \emph{external neighborhood} of $X$. For  disjoint sets $X, Y\subseteq V(G)$, let $e_G(X, Y)$ denote the number of edges in $G$ between $X$ and $Y$. Given $\varepsilon_1>0$ and
$k>0$, define
\begin{equation}\label{eq:rhofunction}
 \rho(x)=\rho(x;\varepsilon_1,k):=
 \begin{cases}
  0,&x<k/5,\\[2mm]
  \displaystyle\tfrac{\varepsilon_1}{\log^2(15x/k)},&x\ge k/5.
 \end{cases}
\end{equation}
We suppress $\varepsilon_1$ and $k$ when they are clear from context. The first notion of the sublinear expander is due to Koml\'{o}s and Szemer\'{e}di.  
\begin{definition}[\cite{KS1,KS2}]\label{def:KS}
A graph $G$ is a \emph{KS $(\varepsilon_1,k)$-expander} if
 $|N_G(X)|\ge \rho(|X|)|X|$
for every $X\subseteq V(G)$ with
$k/2\le |X|\le |V(G)|/2$.
\end{definition}

Haslegrave, Kim and Liu~\cite{HKL} introduced the following robust
version of Definition~\ref{def:KS}. Here $G-F$ denotes the spanning subgraph of $G$ with the edges in $F$ removed.

\begin{definition}\label{defn:HKL}
A graph $G$ is an \emph{$(\varepsilon_1,k)$-robust expander}
if, for every $X\subseteq V(G)$ with
$k/2\le |X|\le |V(G)|/2$ and every $F\subseteq E(G)$ with
$|F|\le d(G)\rho(|X|)|X|$, we have
\(
    |N_{G-F}(X)|\ge \rho(|X|)|X|.
\)
\end{definition}

Every graph contains a robust expander with almost the same average
degree, provided the parameters are chosen as
in~\cite[Lemma~3.2]{HKL}. Letzter, Methuku and Sudakov~\cite{Nearly-H-cycle} defined a weaker robust formulation, with
a uniform polylogarithmic expansion rate for all set sizes.
\begin{definition}\label{defweakro}
An $n$-vertex graph $G$ is a \emph{weak
$(\varepsilon,c,s)$-robust expander} if, for every
$X\subseteq V(G)$ with $1\le |X|\le 2n/3$ and every
$F\subseteq E(G)$ with $|F|\le s|X|$, we have
\(
    |N_{G-F}(X)|\ge \frac{\varepsilon}{\log^c n}|X|.
\)
\end{definition}

The last one is an edge expander, which appeared in both~\cite{Nearly-H-cycle} and  Brada\v{c} and Janzer's work~\cite{Janzer}.

\begin{definition}[\cite{Janzer}]\label{BJ-ga}
An $n$-vertex graph $G$ is a \emph{$\gamma$-edge-expander} if for every $X\subseteq V(G)$ with $1\le |X|\le 2n/3$,
$ e_G(X,\overline X)\ge \gamma d(G)|X|.$
\end{definition}

The relations of these four expanders are depicted in \cref{fig:relations}. We defer a more detailed comparison to Proposition \ref{lem:expansion-relations}. In this paper, we give two constructions showing that regular expanders as in Definitions~\ref{def:KS} and~\ref{BJ-ga} (hence also~\ref{defweakro}) could have small circumference.

\subsection{Counterexamples to Montgomery's conjecture}

Letzter, Methuku and Sudakov \cite{Nearly-H-cycle} developed methods for finding nearly Hamilton paths and cycles in robust sublinear expanders with strong regularity properties. Motivated by this, Montgomery \cite{MontgomerySurvey} proposed
the corresponding Hamiltonicity question for regular KS
expanders.

\begin{conjecture}[\cite{MontgomerySurvey}, Conjecture 5.2]\label{conj}
For every $\varepsilon>0$, there exists $d_0$ such that, for every
$d\ge d_0$, every $d$-regular KS $(\varepsilon,d)$-expander contains a
Hamilton cycle.
\end{conjecture}

We show that this conjecture fails in a strong sense: even at logarithmic-square average degree, a regular KS   expander can have circumference that is a small proportion of its order. %{\color{blue}with the expansion constant chosen sufficiently small depending on the proportion}.
Recall that the \emph{circumference} $\circum(G)$ is the maximum length
of a cycle in $G$.

\begin{theorem}\label{thm:main}
Let $0<\eta<1/2$ and $0<\varepsilon_2<1/20$. There exists
$\varepsilon_0=\varepsilon_0(\eta,\varepsilon_2)>0$ such that, for
every fixed $0<\varepsilon_1<\varepsilon_0$, there is an $n$-vertex $d$-regular graph $G$
such that
\stepcounter{propcounter}
    \begin{enumerate}[label = ({\bfseries \Alph{propcounter}\arabic{enumi}})]
        \rm
      \item\label{thcon1}
 $G$ is a KS $(\varepsilon_1,\varepsilon_2d)$-expander,
\rm \item\label{thcon2}
 $\circum(G)<\eta n$, and
 \rm \item\label{thcon3}
 $d=(1/2+o_d(1))\log^2n$.
\end{enumerate}
\end{theorem}

The obstruction of a large circumference is provided by a sparse vertex separator whose removal leaves many dense blocks. The longest cycle  traverses only as many blocks as the separator can connect. We guarantee $d$-regularity by starting from a biregular Ramanujan graph with bipartition $(L,S)$, where each vertex of $L$ is replaced by $K_{d+1}$ minus a matching of reasonable size and $S$ is an independent set. The density of the blocks supplies expansion within blocks and repairs regularity, while the spectral property of the base graph controls expansion between
the blocks and $S$. This yields the required neighborhood for arbitrary vertex sets. Meanwhile, $S$ remains a vertex separator. Removing the
vertices in $V(C)\cap S$ from any cycle $C$ leaves paths contained within individual blocks. Hence $C$ visits at most $|S|$ blocks, which gives the upper bound of circumference. 

In addition, the order of a union of a positive proportion of the blocks is $\Theta(n)$ while the order of its neighborhood is only $\Theta(n/d)$. Since sublinear expansion at this scale requires a neighborhood of order
$\Theta(n/\log^2(n/d))$, the separator obstruction is compatible
 with expansion only if $d=O\left(\log^2\tfrac nd\right)$. 
Our choice of parameters yields
$d=(1/2+o(1))\log^2n$. We call this the
\emph{block--separator construction}, or simply the \emph{BS construction}.

To see that Theorem \ref{thm:main} gives counterexamples in Montgomery's
normalization, set $\varepsilon_*:=\varepsilon_1
 \left(\tfrac{\log3}{\log(15/\varepsilon_2)}\right)^2$. 
Since $t\mapsto\log t/\log(5t/\varepsilon_2)$ is increasing for
$t\ge3$, for every $x\ge d$ we have
$ \tfrac{\varepsilon_1}{\log^2(15x/(\varepsilon_2d))}
 \ge
 \tfrac{\varepsilon_*}{\log^2(3x/d)}.$
Thus every graph supplied by Theorem \ref{thm:main} is a KS
$(\varepsilon_*,d)$-expander in the sense of Conjecture \ref{conj}.

\subsection{Edge expanders with small circumference}

A well-known result of Krivelevich \cite{Krivelevich} states that every $n$-vertex $d$-regular $\gamma$-edge-expander $G$ satisfies $\circum(G)=\Omega(\gamma n)$. 
Addressing a question of Janzer (private communication), we give a matching construction with $\circum(G)=O(\gamma n)$ and $d=\Theta(\gamma^{-1})$, showing that the order $\Theta(\gamma n)$ of $\circum(G)$ is best possible.

\begin{theorem}\label{thm:edge-main}
There exists an absolute constant $\gamma_0>0$ such that, for every $0<\gamma<\gamma_0$ and infinitely many integers $n>0$, there is an $n$-vertex $d$-regular graph $G$ with $\tfrac{1}{64\gamma}\le d\le \tfrac{1}{32\gamma}$ satisfying
\stepcounter{propcounter}
    \begin{enumerate}[label = ({\bfseries \Alph{propcounter}\arabic{enumi}})]
        \rm
      \item\label{BJ-la1} $G$ is a $\gamma$-edge-expander, and
    \item\label{BJ-la2} $\circum(G)\le64\gamma n$.
\end{enumerate}
\end{theorem}

Here we begin with a regular expander $H$. For each $v\in V(H)$, we attach an almost-complete graph $Q_v$ to $v$ by $\Theta(d)$ edges, and the resulting graph is denoted by $G$. The expansion of $H$ and the density inside $Q_v$ guarantee the edge expansion of $G$, while the single-vertex attachments ensure that every cycle of $G$ lies either in $H$ or in some $V(Q_v)\cup \{v\}$. We call this the \emph{dense--lobe construction}, or simply the \emph{DL construction} and $Q_v$ is called a \emph{lobe}.

\vspace{0.2cm}
\noindent 
\textbf{Organization:} The remainder of the paper is organized as follows. In \cref{sec:preliminaries}, we recall the Ramanujan graph tools used in
both constructions. 
In \cref{sec:bs-construction}, we give the BS construction and prove Theorem \ref{thm:main}. In \cref{sec:dl-construction}, we give the DL construction and prove Theorem \ref{thm:edge-main}. 
We conclude with the remaining robust-expansion and degree-threshold questions in \cref{sec:conc}. 

\section{Preliminaries}\label{sec:preliminaries}
All graphs considered in this paper are simple unless stated otherwise, and all logarithms are natural. 
A bipartite graph $B$ with bipartition $(L,S)$ is $(r,s)$-\emph{biregular} if every vertex in $L$ has degree $r$ and every vertex in $S$ has degree $s$. Obviously, $r|L|=s|S|$. Let $M_B$ denote the $|L|\times |S|$ biadjacency matrix of $B$, whose
$(u,v)$-entry is $1$ if $uv\in E(B)$ and $0$ otherwise. Then
the adjacency matrix of $B$ can be written as
$$
    A_B=
    \begin{pmatrix}
        0 & M_B\\
        M_B^{\mathsf T} & 0
    \end{pmatrix}.
$$
Consequently, the eigenvalues of $A_B$ consist of the positive and negative singular values of $M_B$, together with additional zeros when $|L|\neq |S|$.

Since
$M_B\mathbf{1}_S=r\mathbf{1}_L$ and $M_B^{\mathsf T}\mathbf{1}_L=s\mathbf{1}_S,$
the largest singular value of $M_B$ is $\sqrt{rs}$, which we call the \emph{trivial singular value}. Writing the singular values of $M_B$, with multiplicity, as
$$
    \sqrt{rs}=\sigma_1(M_B)\geq \sigma_2(M_B)\geq\cdots,
$$
we define
$
    \lambda(B):=\sigma_2(M_B).
$
In terms of the adjacency spectrum, the trivial adjacency eigenvalues of $B$ are $\pm\sqrt{rs}$, and every nontrivial adjacency eigenvalue has absolute value at most $\lambda(B)$. Following Marcus, Spielman and Srivastava~\cite{MSS}, we say that $B$ is \emph{Ramanujan} if
\begin{equation}\label{Rama-bound}
    \lambda(B)\leq \sqrt{r-1}+\sqrt{s-1}.
\end{equation}

We briefly recall the $2$-lift construction used in~\cite{MSS}.
Given a graph $B$, a $2$-lift $\widehat{B}$ is obtained by replacing each vertex $v\in V(B)$ with a pair $\{v_0,v_1\}$. For every edge $uv\in E(B)$, its two lifts are chosen to be either
$\{u_0v_0,u_1v_1\}$ or $\{u_0v_1,u_1v_0\}$.
Equivalently, these two choices are encoded by a signing
$\tau\colon E(B)\to\{+1,-1\}$. If $B$ is bipartite with
bipartition $(L,S)$, then every $2$-lift of $B$ is bipartite with
bipartition
$
    \bigl(L\times\{0,1\},\,S\times\{0,1\}\bigr).
$
Moreover, a $2$-lift preserves the two degrees and doubles the sizes of both vertex classes. 
The adjacency eigenvalues of $\widehat{B}$ are the multiset union of the eigenvalues of $B$, called the \emph{old eigenvalues}, and the eigenvalues of the signed adjacency matrix associated with $\tau$, called the \emph{new eigenvalues}. 

Since the old eigenvalues are inherited from the base graph, this makes it possible to construct bipartite Ramanujan graphs inductively by taking suitable $2$-lifts. Start with the complete bipartite graph
$B_0:=K_{d,c},$
viewed as a $(c,d)$-biregular graph with bipartition $(L,S)$, where $|L|=d$ and $|S|=c$.
Its biadjacency matrix has rank $1$. Hence all of its nontrivial singular values are $0$, and therefore $B_0$ is Ramanujan. 
Marcus, Spielman, and Srivastava~\cite[Theorems~5.3 and~5.6]{MSS}
showed that every bipartite Ramanujan graph admits a signing for which all new eigenvalues satisfy the corresponding Ramanujan bound.
Since the new spectrum of a bipartite lift is symmetric about zero,
the resulting \(2\)-lift is also bipartite Ramanujan. 
Iterating this construction $t$ times preserves
the degrees $c$ and $d$ and doubles the sizes of both vertex classes at each step. The resulting graph therefore satisfies
$|L|=d2^t$ and $|S|=c2^t$. Hence, we immediately obtain the following proposition.

\begin{proposition}
\label{lem:Ramanujan-existence}
For all integers $c,d\geq 3$ and $t\geq 0$, there exists a $(c,d)$-biregular bipartite Ramanujan graph $B$ with bipartition
$(L,S)$ such that
    $|L|=d2^t$
 and 
    $|S|=c2^t.$
\end{proposition}

We will also use the discrete Cheeger inequality and the standard expander mixing lemma for biregular bipartite graphs.

\begin{lemma}[Discrete Cheeger inequality
{\cite[Theorem~2.4]{HLW}}]
\label{lem:discrete-cheeger}
Let $G$ be a $d$-regular graph, and let $\lambda_2(G)$ denote its
second largest adjacency eigenvalue. Then, for every
$X\subseteq V(G)$,
\[
 e_G(X,V(G)\setminus X)
 \ge
 \tfrac{d-\lambda_2(G)}{2}
 \min\{|X|,|V(G)|-|X|\}.
\]
\end{lemma}

\begin{lemma}[Biregular expander mixing lemma {\cite[Lemma~8]{DeWinter-Schillewaert-Verstraete}}]
\label{lem:biregular-mixing}
Let $B$ be an $(r,s)$-biregular bipartite graph with bipartition
$(L,S)$. Then for every $A\subseteq L$ and $Z\subseteq S$, we have
$$
\left|
    e_B(A,Z)-\tfrac{r|A||Z|}{|S|}
\right|
\leq
\lambda(B)\sqrt{|A||Z|}.
$$
Since $r|L|=s|S|$, the main term can equivalently be written as
$
    \tfrac{r|A||Z|}{|S|}
    =
    \tfrac{s|A||Z|}{|L|}.
$
\end{lemma}

\section{The block--separator construction}\label{sec:bs-construction}
In this section, we first describe our block--separator graph construction $G_\mathrm{BS}$ in terms of the parameters \(\beta,d,m\), and specify their exact values.  
Then, in \cref{sec:bs-proof}, we verify the KS expansion of $G_\mathrm{BS}$ and small circumference, which proves Theorem~\ref{thm:main}.

\subsection{Construction and choice of parameters}
\begin{definition}
   Let $G_\mathrm{BS}=G(\beta, d, m)$ be a graph as follows. We begin with an auxiliary \((\beta d,d)\)-biregular bipartite Ramanujan graph $B$ with bipartition \((L,S)\), where
    $|L|=m$ and $|S|=\beta m$. 
    For each vertex \(i_j\in L\), replace \(i_j\) by a copy of $H_j:=K_{d+1}-M_j,$
where \(M_j\) is a matching of size \(\beta d/2\). Call the \(\beta d\) endpoints of \(M_j\) the \emph{ports} of \(H_j\), and identify them bijectively with the edges of \(B\) incident with \(i_j\). For every \(i_js\in E(B)\), join \(s\) to the port of \(H_j\) corresponding to the edge $i_js$, and keep \(S\) independent; see Figure \ref{fig:def3.1}.
\end{definition}

For the construction to be well-defined, we only need \(\beta d/2\) to be an integer and \(\beta d\le d+1\).
In the rest of the paper, we set 
\begin{equation}\label{eq:beta-value}
\beta:= 2\left\lfloor\tfrac{\eta d}{8}\right\rfloor/d.
\end{equation}
For \(0<\eta<1/2\) and all sufficiently large
\(d\), we have
\begin{equation}\label{eq:beta-range}
\tfrac \eta8\le \beta\le \tfrac{\eta}{4}<\tfrac18.
\end{equation}
Throughout the remaining paper, we define
\begin{equation}\label{eq:alpha-delta}
(\alpha,\delta,\varepsilon_0):=
(2^{-8}\eta^2,2^{-10}\eta^2,2^{-10}\eta^2).
\end{equation}
\begin{figure}[H]
\begin{center}
      \includegraphics[scale=0.32]{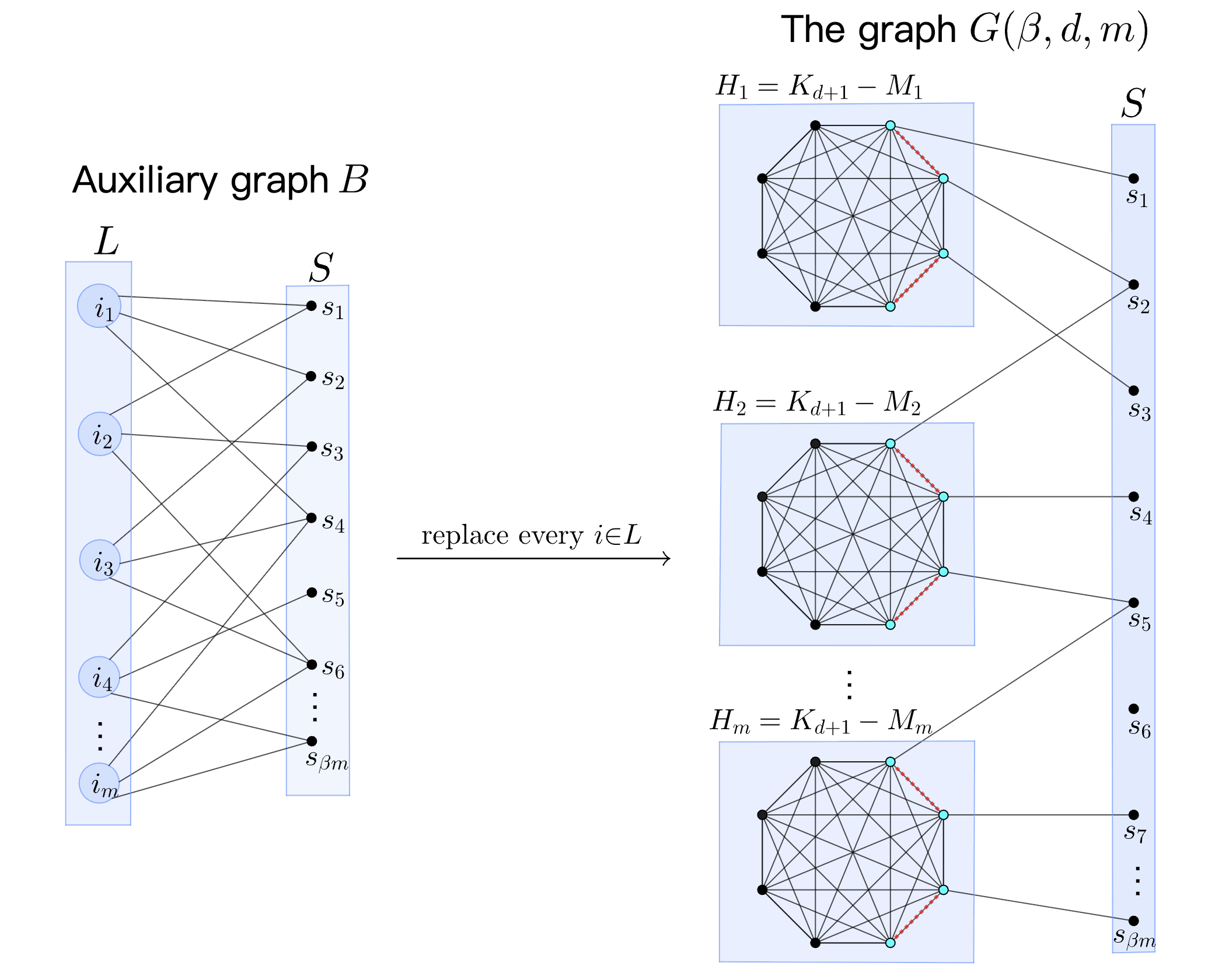}\\
      \caption{The auxiliary graph \(B\) with bipartition \((L,S)\). Each \(i_j\in L\) is replaced by \(H_j=K_{d+1}-M_j\), and the vertices of \(S\) are joined to the corresponding ports.}
     \label{fig:def3.1}
    \end{center}
  \end{figure}
\vspace{-0.7cm}
\begin{lemma}\label{lem:construction-existence}
Let \(d\) be a sufficiently large positive integer and \(\beta\) be as in (\ref{eq:beta-value}).
For every integer \(t\ge 0\), set $m:=d2^t$. Then the graph \(G_\mathrm{BS}= G(\beta,d,m)\) exists.
Moreover, $G_\mathrm{BS}$ is \(d\)-regular with
$|V(G_\mathrm{BS})|=m(d+1+\beta)$. The auxiliary graph $B$ satisfies
\begin{equation}\label{eq:spectral}
\lambda(B)
\le \sqrt{\beta d-1}+\sqrt{d-1}
\le 1.7\sqrt d.
\end{equation}
\end{lemma}

\begin{proof}
Apply Proposition \ref{lem:Ramanujan-existence} with \(c=\beta d\).
Since \(m=d2^t\), there exists a \((\beta d,d)\)-biregular
bipartite Ramanujan graph \(B\) with bipartition \((L,S)\) such that
$|L|=d2^t=m$ 
and 
$|S|=\beta d2^t=\beta m.$
Since \(B\) is Ramanujan,
$\lambda(B)
\le \sqrt{\beta d-1}+\sqrt{d-1}
\le 1.7\sqrt d.$

As \(\beta d/2\) is an integer and \(\beta<1\), there is a matching of size \(\beta d/2\) in \(K_{d+1}\), so the construction of every
block \(H_j\) is well-defined. 
Every non-port vertex has degree \(d\) within its block. Every port
has degree \(d-1\) within its block and one neighbor in \(S\), while
every vertex of \(S\) has degree \(d\). Thus, \(G_\mathrm{BS}\) is 
\(d\)-regular. Finally,
$
    |V(G_\mathrm{BS})|=m(d+1)+|S| =m(d+1+\beta).
$
\end{proof}

Applying Lemma \ref{lem:biregular-mixing} to the auxiliary graph $B$ from Lemma \ref{lem:construction-existence} yields
\begin{equation}\label{eq:mixing}
\left|e_B(A,Z)-\tfrac{d|A||Z|}{m}\right|
\le 1.7\sqrt{d|A||Z|}
\qquad
(A\subseteq L,\ Z\subseteq S).
\end{equation}
Furthermore, we choose $t$ as the largest integer \(t\ge 0\) such that
\begin{equation}\label{eq:choice_t}
    \log\bigl(d\,2^t(d+1+\beta)\bigr)=\log n\le \sqrt{2d}.
\end{equation}
For all sufficiently large \(d\), the choice \(t=0\) satisfies
\eqref{eq:choice_t}, and so such a largest integer \(t\) exists. Under the choice of $t$, we have
$\sqrt{2d}-\log 2<\log n\le \sqrt{2d},$
and hence
\begin{equation}\label{d-value}
d=\left(\tfrac12+o_d(1)\right)\log^2 n,
\end{equation}
satisfying \ref{thcon3} of Theorem \ref{thm:main}.
As $m=d\,2^t$ and $n=m(d+1+\beta)$, we have that for all sufficiently large \(d\),
$\log\!\left(\tfrac{15n}{2\eps_2d}\right)
=(1+o(1))\sqrt{2d}
,$
 and hence
\begin{equation}\label{eq:scale}
\rho\left(\tfrac{n}{2}\right)\tfrac{n}{2}=\tfrac{\eps_1n}{2\log^2(15n/(2\eps_2d))}\le \tfrac{\eps_1 n}{4(1+o(1))d}\leq \tfrac{\varepsilon_1m}{2(1+o(1))}
<\delta m.
\end{equation}
The penultimate inequality holds as $\eps_1\le \eps_0=\delta$ and $n=m(d+1+\beta)\le 2md$.
%In \cref{sec:bs-proof},
Next, we use \eqref{eq:scale} to ensure that the
expansion required for a set of size \(n/2\) is less than \(\delta m\);
see \eqref{eq:linear-scale} below.

\subsection{$G_\mathrm{BS}$ is a KS expander}\label{sec:bs-proof}
%Before proceeding with our proof, we 
We first prove the following proposition.

\begin{proposition}\label{prop:spectral-neighbourhood}
Let \(\beta\) and \(\alpha\) be chosen as in~(\ref{eq:beta-value}) and~(\ref{eq:alpha-delta}).
Suppose that \(B\) is a
$(\beta d,d)$-biregular bipartite graph with bipartition \((L,S)\) such
that $|L|=m$, $|S|=\beta m$ and $\lambda(B)\le 1.7\sqrt d$. Then every
\(A\subseteq L\) satisfies
$|N_B(A)|\ge \alpha\min\{d|A|,m\}.$
\end{proposition}

\begin{proof}
Let \(Z:=N_B(A)\). Since every edge incident with \(A\) ends in \(Z\), by \eqref{eq:mixing} we have
\[
\beta d|A|
\le \tfrac{d|A||Z|}{m}+1.7\sqrt{d|A||Z|}.
\]
Suppose for contradiction that
\(|Z|<\alpha\min\{d|A|,m\}\). Recall that \(0<\eta<1/2\) and $\alpha=2^{-8}\eta^2$, and then the right-hand side of the above inequality is less than
\[
\bigl(\alpha+1.7\sqrt\alpha\bigr)d|A|
\le \left(2^{-8}\eta^2+1.7\cdot 2^{-4}\eta\right)d|A|
<\tfrac \eta8d|A|
\le\beta d|A|.
\]
This is a contradiction. 
\end{proof}

\begin{proof}[Proof of Theorem \ref{thm:main}]
Fix $0<\varepsilon_1<\varepsilon_0$, and take
$G=G(\beta,d,m)$ with the parameters chosen above and $m=d2^t$, where
$t$ is chosen as in \eqref{eq:choice_t}. 
%By construction, $G$
By Lemma~\ref{lem:construction-existence}, such $G$ exists and 
is an $n$-vertex $d$-regular graph. Moreover, by \eqref{d-value}, 
$d=(1/2+o_d(1))\log^2 n$, which in particular proves \ref{thcon3}.
It remains to prove \ref{thcon1} and \ref{thcon2}. We begin with the
circumference bound.

Let \(C\) be a cycle and set $k:=|V(C)\cap S|$. Recall that \(S\) is
independent and that \(G-S\) is the disjoint union of the blocks
\(H_j\). If \(k=0\), then \(C\) lies in one block and has length at
most \(d+1<\eta n\) for all sufficiently large \(d\). Suppose that
\(k>0\). Deleting \(C\cap S\) leaves \(k\) paths, each contained in one
block, and hence \(C\) visits at most \(k\) blocks. Since
\(k\le |S|=\beta m\), we have
\[
\tfrac{|C|}{n}
\le \tfrac{\beta(d+2)}{d+1+\beta}
<\eta
\]
for all sufficiently large \(d\). Thus \ref{thcon2} holds.

In the following, %we mainly 
it suffices to
complete the proof of \ref{thcon1}. Let \(X\subseteq V(G)\) be a set  with $\tfrac{\eps_2d}{2}\le |X|\le \tfrac n2$. Set $Y:=X\cap S$, and let \(A\subseteq L\) index the blocks met by \(X\). For any $j\in [m]$ such that $i_j\in A$, set $X_j:=X\cap V(H_j)$. Denote by \(p\) the number of vertices \(i_j\in A\) with \(|X_j|=1\), and set
\[
q:=\sum_{i_j\in A}(d+1-|X_j|).
\]
Thus \(q\) is the total number of vertices not included by $X$ from the blocks indexed by \(A\).
Set
\[
U:=N_G(X)\cap \left(\bigcup_{i_j\in A}V(H_j)\right).
\]
If \(|X_j|=1\), then the unique vertex of \(X_j\) has at least \(d-1\) neighbors inside \(H_j\). If \(|X_j|\ge 2\), then every vertex of \(V(H_j)\setminus X_j\) has a neighbor in \(X_j\) as every vertex of \(H_j\) has at most one non-neighbor in that block. Therefore
\begin{equation}\label{eq:U}
|U|\ge q-p.
\end{equation}

We first record the maximum expansion required in the testing range. Recall that $\delta=2^{-10}\eta^2$. Since $\rho(x)x$ is increasing in $x$ for \(x\ge \eps_2d/2\), \eqref{eq:scale} gives
\begin{equation}\label{eq:linear-scale}
\rho(|X|)|X|
\le \rho(n/2)\tfrac n2
<\delta m.
\end{equation}
Thus, to prove the required expansion property for $X$, it suffices to show 
\begin{equation}\label{eq:linear_m}
    |N_G(X)|\ge \delta m.
\end{equation}
This estimate will be useful when $|A|/m$ is large.

 \begin{claim}\label{claim:Alarge}
      If $|A|>\tfrac{3m}{5}$, then \eqref{eq:linear_m} holds and hence \ref{thcon1} holds.
 \end{claim}
 \begin{poc}
    As \(|X|\le n/2=m(d+1+\beta)/2\), we have
\[
d|A|-|X|
>m\left(\tfrac{3d}{5}-\tfrac{d+1+\beta}{2}\right)
>\tfrac m{40}\ge\delta m
\]
for all sufficiently large \(d\). Since \(p\le |A|\), \eqref{eq:U} gives
\[
|N_G(X)|
\ge |U|
\ge q-p
\ge (d+1)|A|-|X|-|A|
=d|A|-|X|
>\delta m. \qedhere
\] 
%as claimed.
 \end{poc}

We may therefore assume from now on that $|A|\le \tfrac{3m}{5}$. 
In this case we also use the neighbors across the interface. They come from two directions: vertices of \(Y\) send edges into untouched blocks, while the touched blocks expose the vertices of \(S\setminus Y\). Accordingly, define
\[
\Phi(A,Y):=
e_B(Y,L\setminus A)+|N_B(A)\setminus Y|.
\]
Next, we give a lower bound on \(\Phi(A,Y)\).

\begin{claim}\label{clm:interface}
If $|A|\le \tfrac{3m}{5}$, then $\Phi(A,Y)\ge 2\delta\min\{d(|A|+|Y|),m\}$.
\end{claim}

\begin{poc}
Since every vertex of \(Y\) has degree \(d\) in \(B\), we have
\[
e_B(Y,L\setminus A)=d|Y|-e_B(A,Y).
\]
Applying \eqref{eq:mixing} to the last term, we have
\begin{equation}\label{eq:outside-edges}
e_B(Y,L\setminus A)
\ge d|Y|-\tfrac{d|A||Y|}{m}
-1.7\sqrt{d|A||Y|}.
\end{equation}
To estimate the other term $|N_B(A)\backslash Y|$ in $\Phi(A,Y)$, we use Proposition \ref{prop:spectral-neighbourhood} and divide it into two cases depending on the behavior of $\min\{d|A|,m\}$. Recall that $\alpha=2^{-8}\eta^2$.

\textbf{Case 1: \(d|A|\le m\).}
By Proposition  \ref{prop:spectral-neighbourhood}, $|N_B(A)\setminus Y|
\ge \alpha d|A|-|Y|$. 
Together with \eqref{eq:outside-edges}, we have
\[
\Phi(A,Y)
\ge
\alpha d|A|+(d-2)|Y|
-1.7\sqrt{d|A||Y|}.
\]
By the arithmetic--geometric mean inequality
$2\sqrt{|A||Y|}\le |A|+|Y|$ and the fact that
$\tfrac{1.7\sqrt d}{2}\le\tfrac{\alpha d}{4}$, 
we have that for all sufficiently large \(d\),
$1.7\sqrt{d|A||Y|}
\le \tfrac{\alpha d}{4}(|A|+|Y|).$
Consequently,
\[
\Phi(A,Y)
\ge \tfrac{\alpha d}{2}(|A|+|Y|)
=2\delta d(|A|+|Y|).
\]

\textbf{Case 2: \(d|A|> m\).}
By Proposition \ref{prop:spectral-neighbourhood},
$|N_B(A)|\ge \alpha m.$
Recall that $\delta=\alpha/4$. 
If \(|Y|\le\alpha m/2\), then
\[
\Phi(A,Y)\ge |N_B(A)|-|Y|
\ge\tfrac{\alpha m}{2}= 2\delta m.
\]
Otherwise \(|Y|>\alpha m/2\). 
Since \(|A|\le3m/5\), the first two terms in
\eqref{eq:outside-edges} are at least \(2d|Y|/5\). Moreover,
\[
\tfrac{1.7\sqrt{d|A||Y|}}{d|Y|}
\le
1.7\sqrt{\tfrac{3m}{5d|Y|}}
\le
1.7\sqrt{\tfrac{6}{5\alpha d}}
=o_d(1),
\]
where we used \(|Y|>\alpha m/2\). Hence, for all sufficiently large
\(d\), we have $e_B(Y,L\setminus A)\ge\tfrac{d|Y|}{5}$. 
This completes the proof of Claim \ref{clm:interface}.
\end{poc}
 \vspace{-0.5cm}
\begin{figure}[H]
\begin{center}
\includegraphics[scale=0.65]{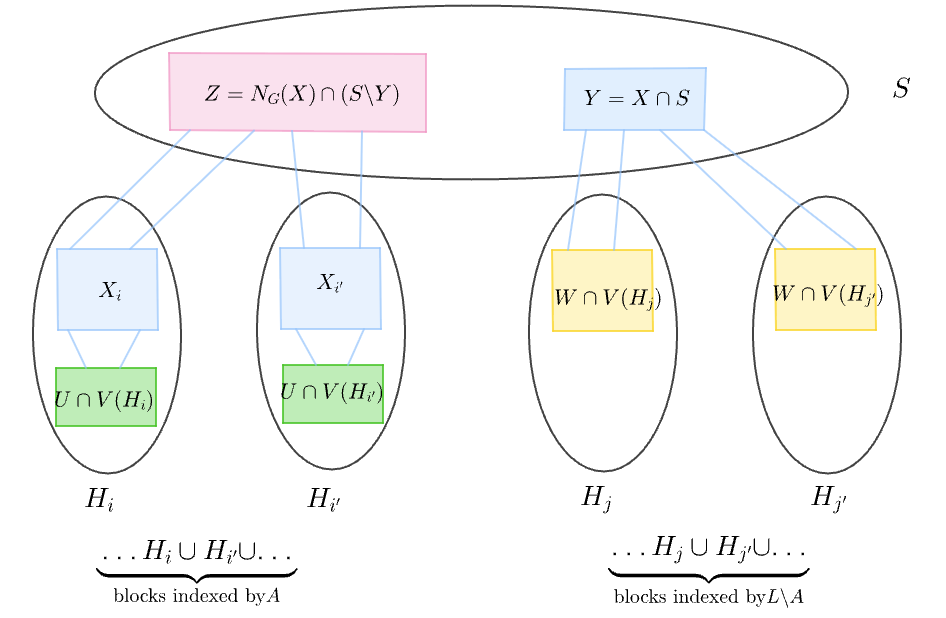}\\
      \caption{Illustration of the  decomposition of $N_G(X)=U\mathbin{\cup}W\mathbin{\cup}Z$. Here $U$ comes from internal neighbors in the blocks indexed by $A$, while the blue lines represent the edges leaving $X$.
}\label{fig:prof1}
\end{center}
\end{figure}
\vspace{-0.6cm} 
We now compare \(\Phi(A,Y)\) with \(N_G(X)\). Set
\[
W:=N_G(X)\cap \left(\bigcup_{i_j\notin A}V(H_j)\right)
\qquad\text{and}\qquad
Z:=N_G(X)\cap(S\setminus Y).
\]
The sets \(U,W,Z\) are pairwise disjoint; see Figure \ref{fig:prof1}. By the definition of ports, we have 
$|W|=e_B(Y,L\setminus A)$. 
Moreover, every vertex of \(N_B(A)\setminus(Y\cup Z)\) has an incident edge from \(A\) whose corresponding port is omitted from \(X\). For each vertex of \(N_B(A)\setminus(Y\cup Z)\), choose one incident
edge from \(A\). The corresponding omitted ports are distinct for
distinct vertices of \(S\). Hence $|N_B(A)\setminus(Y\cup Z)|\le q$, and therefore
$|Z|\ge |N_B(A)\setminus Y|-q$. Combining this with \eqref{eq:U}, we obtain
\[
|N_G(X)|
\ge |U|+|W|+|Z|
\ge(q-p)+e_B(Y,L\setminus A)
+|N_B(A)\setminus Y|-q
=\Phi(A,Y)-p.
\]

On the other hand, the \(p\) singleton blocks have disjoint internal neighborhoods, each of size at least \(d-1\). Hence
$|N_G(X)|\ge(d-1)p.$
Consequently,
\begin{equation}\label{eq:master}
|N_G(X)|
\ge
\max\{\Phi(A,Y)-p,(d-1)p\}
\ge\left(1-\tfrac1d\right)\Phi(A,Y)
\ge\tfrac12\Phi(A,Y).
\end{equation}

\iffalse
\begin{figure}[H]
\centering
\begin{tikzpicture}[every node/.style={font=\small},>=Latex]
  \node[draw,rounded corners,minimum width=2.9cm,minimum height=1cm,
        align=center] (blocks) at (0,0.9) {$X$ inside the\\blocks indexed by $A$};
  \node[draw,rounded corners,minimum width=2.9cm,minimum height=1cm,
        align=center] (sep) at (0,-1.0) {$Y=X\cap S$};
  \node[draw,ellipse,minimum width=2.4cm,minimum height=0.8cm]
        (U) at (5,1.8) {$U$: internal};
  \node[draw,ellipse,minimum width=2.4cm,minimum height=0.8cm]
        (Z) at (5,0.3) {$Z\subseteq S\setminus Y$};
  \node[draw,ellipse,minimum width=2.4cm,minimum height=0.8cm]
        (W) at (5,-1.2) {$W$: untouched blocks};
  \draw[->] (blocks) -- (U);
  \draw[->] (blocks) -- (Z);
  \draw[->] (sep) -- (W);
\end{tikzpicture}
\caption{The three disjoint contributions
$N_G(X)\supseteq U\mathbin{\dot\cup}W\mathbin{\dot\cup}Z$ used in the
BS expansion proof.}
\label{fig:prof1}
\end{figure}
\fi
  
By Claim~\ref{clm:interface} and \eqref{eq:master}, $|N_G(X)|
\ge\delta\min\{d(|A|+|Y|),m\}$. If $\min\{d(|A|+|Y|),m\}=m$, then \eqref{eq:linear_m} and \eqref{eq:linear-scale} give \ref{thcon1}.
Otherwise,
$|X|\le(d+1)|A|+|Y|
\le2d(|A|+|Y|),$
and hence $|N_G(X)|\ge\tfrac{\delta}{2}|X|$. 
Since \(|X|\ge\eps_2d/2\),
$\rho(|X|)
\le \tfrac{\eps_1}{\log^2(15/2)}
<\tfrac{\delta}{2}$,
and therefore
$|N_G(X)|\ge \rho(|X|)|X|.$
This completes the proof of Theorem \ref{thm:main}.
\end{proof}

\section{The dense--lobe construction}\label{sec:dl-construction}

This section is devoted to the proof of Theorem \ref{thm:edge-main}. We first use the following classical result to give a lower bound on the circumference of edge expanders; see Proposition \ref{prop:generic-lower}.
\subsection{Construction and basic properties}
\begin{lemma}[\cite{Krivelevich}]\label{thm_krivel}
    Let $k>0, t\ge 2$ be integers. Let $G=(V, E)$ be a graph on more than $k$ vertices satisfying  $|N_G(W)|\ge t$ for every $W\subset V$ with $k/2\le |W|\le k$, then $\circum(G)\ge t+1$.
\end{lemma}

\begin{proposition}\label{prop:generic-lower}
Let $n, d$ be positive integers, and let $\gamma>0$ satisfy
%Suppose $n\in \mathbb{N}$ and $\gamma>0$ with 
$\gamma n\geq 10$. Then every $n$-vertex $d$-regular $\gamma$-edge-expander $G$ satisfies
$ \circum(G)\ge\tfrac{\gamma n}{4}-O_{\gamma}(1).$
\end{proposition}

\begin{proof}
As $G$ is $d$-regular, for any $v\in V(G)$, $d=e_G(\{v\}, V(G)\backslash\{v\})\ge \gamma d$ and hence $\gamma\le 1$. 
Set $k:=\lfloor n/2\rfloor$. For every
$W\subseteq V(G)$ with $k/2\le |W|\le k$, we have
$ |N_G(W)|
 \ge\tfrac{e_G(W,V(G)\setminus W)}{d}
 \ge\gamma|W|
 \ge\tfrac{\gamma k}{2}.$
Set $t=\lfloor\tfrac{\gamma k}{2}\rfloor$. As $\gamma n\ge 10$ and $\gamma\le 1$, $t\ge \lfloor\tfrac{\gamma n}{4}-\tfrac{\gamma}{2}\rfloor\ge 2$.  
Applying Lemma~\ref{thm_krivel} with $t$, we get 
$\circum(G)\geq \lfloor\tfrac{\gamma k}{2}\rfloor+1\geq \tfrac{\gamma n}{4}-\tfrac{\gamma}{4}.$
\end{proof}

The following lemma shows that there exists a regular graph $H$ with a good edge expansion property. Later we will use such $H$ as the base of our construction.
%The expanding core can be chosen with a controlled sequence of orders. This will also allow us to align the degree with $\log^2 n$.

\begin{lemma}\label{lem:dl-core}
For every sufficiently large integer $r$ and every integer $t\ge0$,
there is an $r$-regular bipartite graph $H$ on $r2^{t+1}$ vertices such that for every $A\subseteq V(H)$,
\begin{equation}\label{eq:dl-core-expansion}
 e_H(A,V(H)\setminus A)
 \ge\tfrac r4\min\{|A|,r2^{t+1}-|A|\}.
\end{equation}
\end{lemma}

\begin{proof}
Applying Proposition \ref{lem:Ramanujan-existence} with $(c, d)=(r, r)$, we get an $r$-regular bipartite Ramanujan graph $H$ such that each of the parts has $r2^t$ vertices. Therefore, $|V(H)|=r2^{t+1}$. Denote its
adjacency eigenvalues by $r=\lambda_1\ge\lambda_2\ge\cdots\ge\lambda_{r2^{t+1}}=-r$.
By (\ref{Rama-bound}), $\lambda_2\le2\sqrt{r-1}\le r/2$ for all sufficiently large $r$. Let $A\subseteq V(H)$. Applying Lemma~\ref{lem:discrete-cheeger}, we have
\[
 e_H(A,V(H)\setminus A)
 \ge\tfrac{r-\lambda_2}{2}\min\{|A|,r2^{t+1}-|A|\}\ge \tfrac r4\min\{|A|,r2^{t+1}-|A|\}.\qedhere
\]
%Since $r-\lambda_2\ge r/2$ and $|A|(M-|A|)/M\ge a/2$, we obtain \eqref{eq:dl-core-expansion}.
\end{proof}

\begin{definition}\label{def:dl-construction}
Let $d$ be a positive integer divisible by $4$, and let $H$ be an auxiliary $d/2$-regular graph satisfying Lemma \ref{lem:dl-core} with $r=d/2$ and any $t$. 
For each \(v\in V(H)\), take a copy
$Q_v:=K_{d+1}-\mathcal M_v$,
where \(\mathcal M_v\) is a matching of size \(d/4\), such that these
copies are pairwise vertex-disjoint and disjoint from $V(H)$. Next, join
\(v\) to every vertex of \(\mathcal M_v\).
%For all $v\in V(H)$, there are vertex-disjoint $Q_v:=K_{d+1}-\mathcal M_v$, where $\mathcal M_v$ is a matching of size $d/4$, and connect $v$ to every vertices of $\mathcal M_v$.
We call the resulting graph the \emph{dense--lobe graph associated
with $H$}, and denote it by
 $G_{\mathrm{DL}}=G(d, t)$.
The graph $H$ is called the \emph{core}, and the graphs $Q_v$ are called its \emph{lobes}. 
\end{definition}
\vspace{-0.5cm}
\begin{figure}[H]
\begin{center}
\includegraphics[scale=0.4]{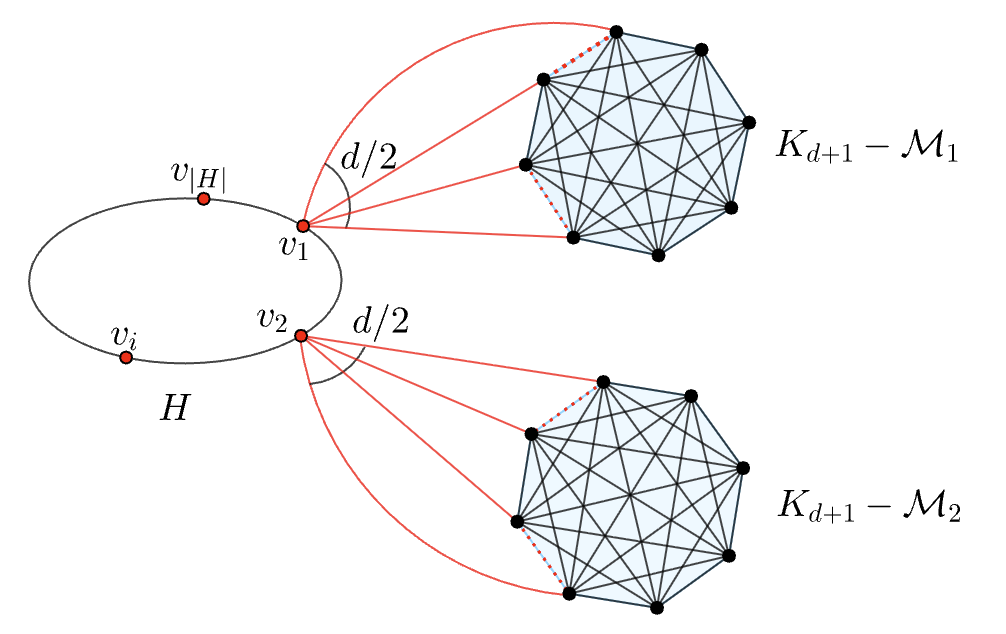}\\
      \caption{
      The DL construction, showing two core vertices and their attached lobes. For each $v\in V(H)$, all edges joining $Q_v$ to the core are incident with $v$.
}\label{fig:dl-construction}
\end{center}
\end{figure}
\vspace{-1cm}
\begin{lemma}\label{lem:dl-regular-pro}
The graph $G_{\mathrm{DL}}=G(d, t)$ satisfies the following properties:
\begin{enumerate}[label=\textup{(\arabic*)}]
\item\label{DL-pro-1} $|V(G_{\mathrm{DL}})|=d2^t(d+2)$;
\item\label{DL-pro-2} $G_{\mathrm{DL}}$ is $d$-regular; and
\item\label{DL-pro-3} $\circum(G_{\mathrm{DL}})\le\max\{d\,2^t,d+2\}.$
\end{enumerate}
\end{lemma}

\begin{proof}
For \ref{DL-pro-1}, we have
$    |V(G_{\mathrm{DL}})|=2^{t+1}(d+1)d/2+2^{t+1}d/2=d2^t(d+2). $
Note that each vertex of $\mathcal M_v$ loses one degree when the matching is deleted from $K_{d+1}$ and obtains one degree from $v$, and each of the remaining vertices of $Q_v$ already has degree $d$. The core vertex $v$ has $d/2$ neighbors in $H$ and $d/2$ neighbors in $Q_v$, so its degree is also $d$. Hence, \ref{DL-pro-2} holds.

Observe that every cycle $C$ in $G_{\mathrm{DL}}$ is contained either in $H$ or in $Q_v\cup\{v\}$ for some $v\in V(H)$. This is because all edges leaving $Q_v$ are incident with the same vertex $v$. Therefore $|C|\le \max\{|V(H)|,d+2\}$ and hence \ref{DL-pro-3} holds.
\end{proof}

\subsection{$G_{\mathrm{DL}}$ is an edge expander}
\label{sec:dl-expansion}
\begin{lemma}\label{thm:dl-explicit}
For every sufficiently large integer $d$ divisible by $4$ and every
integer $t\ge1$, 
$G_{\mathrm{DL}}=G(d, t)$ is a $\gamma_d$-edge-expander with
 $\gamma_d:=\tfrac{1}{32d}$.
\end{lemma}
\begin{proof}
For convenience, write $G=G_{\mathrm{DL}}$ and $n:=|V(G)|$.
We construct an auxiliary multigraph from $G$ by contracting $Q_v$ to a single vertex $q_v$ for each $v\in V(H)$ such that $q_v$ is adjacent to $v$ by $d/2$ parallel edges, and the resulting multigraph is denoted by $\overline{G}$. We define the following weight function for each vertex in $\overline{G}$.
\begin{align*}
  \begin{split}
    \mu(v):=\left\{
      \begin{array}{ll}
        1    & \text{if} \ v\in V(H), \\
        d+1   & \text{if} \ v\in V(\overline{G})\backslash V(H).
      \end{array}
      \right.
    \end{split}
  \end{align*}
Moreover, for $U\subseteq V(\overline{G})$, let $\mu(U):=\sum_{v\in U}\mu(v)$.

\begin{claim}\label{lem:dl-skeleton}
For every $U\subseteq V(\overline{G})$, we have
\begin{equation}\label{eq:dl-skeleton-expansion}
 e_{\overline{G}}(U,V(\overline{G})\setminus U)
 \ge\tfrac1{16}\min\{\mu(U),n-\mu(U)\}.
\end{equation}
\end{claim}

\begin{poc}
Denote
$ A:=U\cap V(H)$ and  
$B:=\{v\in V(H):q_v\in U\}$. For $v\in V(H)$, we say the pair $(v,q_v)$ is \emph{good} if $|\{v,q_v\}\cap U|=1$. Note that a good pair contributes $d/2$ outgoing edges from $U$. Let $q$ be the number of good pairs, and then $q=|A\triangle B|$.
Set $a:=\min\{|A|,|V(H)|-|A|\}$. Therefore
\begin{equation*}%\label{eq:dl-skeleton-cut}
 e_{\overline{G}}(U,V(\overline{G})\setminus U)
 \geq e_H(A,V(H)\setminus A)+\tfrac d2\cdot q
 \ge\tfrac d8a+\tfrac d2q,
\end{equation*}
where the last inequality holds by \eqref{eq:dl-core-expansion}.
Next we compare $\mu(U)$ and $\mu(V(\overline{G})\backslash U)$.
Notice that 
\[
\mu(U)\le (d+2)|A\cap B|+(d+1)q\le (d+2)|A|+(d+1)q.
\]
Similarly,
\[
n-\mu(U)= \mu(V(\overline{G})\backslash U)
\le (d+2)|(V(H)\backslash A)\cap(V(H)\backslash B)|+(d+1)q\le (d+2)(|V(H)|-|A|)+(d+1)q.
\]
Combining these two inequalities above together, we have
\begin{equation*}
\min\{\mu(U),n-\mu(U)\}=
 \min\{\mu(U),\mu(V(\overline{G})\backslash U)\}
 \le(d+2)a+(d+1)q\leq 2da+8dq,
\end{equation*}
where the last inequality holds as $d\ge4$, as claimed.
\end{poc}

\begin{claim}\label{lem:dl-edge-expansion}
For every $X\subseteq V(G)$ with $1\le |X|\le2n/3$, we have
$ e_G(X,V(G)\setminus X)\ge\tfrac{|X|}{32}.$
\end{claim}

\begin{poc}
We first consider the case when $1\leq|X|\le n/2$. 
For each $v\in V(H)$, let $x_v=|X\cap Q_v|$ and let $I=\{v\in V(H): x_v>(d+1)/2\}$. Define $z_v=\min\{x_v,d+1-x_v\}$ and $z=\sum_{v\in V(H)}z_v$. %Note that $z_v=d+1-x_v$ if and only if $v\in I$. 
Let $U=(X\cap V(H))\cup\{q_v: v\in I\}$ be a subset of $V(\overline{G})$. Then
\[
|\mu(U)-|X||=\left|(d+1)|I|-\sum_{v\in V(H)}x_v\right|\le \sum_{v\in I}(d+1-x_v)+\sum_{v\notin I}x_v=\sum_{v\in V(H)}z_v=z.
\]
In particular, $\mu(U)\ge |X|-z$. Moreover,
$n-\mu(U)\ge n-|X|-z\ge |X|-z$ as $|X|\leq n/2$.
Hence,
\begin{equation}\label{eq:dl-rounding-weight}
 \min\{\mu(U),n-\mu(U)\}\ge |X|-z.
\end{equation}

We now estimate the outgoing edges from $X$ in each lobe $Q_v$. Here we naturally partition $V(K_{d+1})$  into two parts such that one part is $X\cap Q_v$ and the other part is $V(K_{d+1})\backslash (X\cap Q_v)$, where the number of vertices in the smaller part is $z_v$. Observe that the smaller part contains $z_v(d+1-z_v)$ outgoing edges, while deleting the matching
$\mathcal M_v$ from $K_{d+1}$ results in the loss of at most $z_v$ of these outgoing edges. Since
$z_v\le d/2$, we obtain $$e_{Q_v}(X\cap Q_v,Q_v\setminus X)
 \ge z_v(d+1-z_v)-z_v
 =z_v(d-z_v)
 \ge\tfrac d2z_v.$$ Summing over all lobes $Q_v$, we have 
\begin{equation}\label{eq:dl-lobe-cut}
 \sum_{v\in V(H)} e_{Q_v}(X\cap Q_v,Q_v\setminus X)\ge dz/2.
\end{equation}
For each $v\in V(H)$, the number of attachment edges crossing the cut in $\overline{G}$ differs from the corresponding number in $G$ by at most $z_v$, since each of the $z_v$ vertices whose membership is changed is incident with at most one attachment edge. After accounting for the total discrepancy, our estimate loses at most $z$ outgoing edges of $X$. Thus, we have
$$e_G(X,V(G)\setminus X)\ge e_{\overline{G}}(U,V(\overline{G})\setminus U)-z+ \sum_{v\in V(H)} e_{Q_v}(X\cap Q_v,Q_v\setminus X).$$ 
By~\eqref{eq:dl-skeleton-expansion}, \eqref{eq:dl-rounding-weight} and \eqref{eq:dl-lobe-cut}, we have
$e_G(X,V(G)\setminus X)
 \ge\tfrac d2z-z+\tfrac{|X|-z}{16}
 \ge\tfrac{|X|}{16}$, 
where the last inequality holds for $d\ge3$.

Now we consider the case where $n/2<|X|\le2n/3$. Denote $Y:=V(G)\setminus X$. Then $|X|/2\leq n/3\le |Y|<n/2$, and the first case implies that
\[
 e_G(X,V(G)\setminus X)
 =e_G(Y,V(G)\setminus Y)
 \ge\tfrac{|Y|}{16}
 \ge\tfrac{|X|}{32}. \qedhere
\]
\end{poc}
Hence, $G_{\mathrm{DL}}$ is a $\tfrac{1}{32d}$-edge-expander.
\end{proof}

\begin{proof}[Proof of Theorem \ref{thm:edge-main}]
 %Combining Lemma \ref{lem:dl-regular-pro} and Lemma \ref{thm:dl-explicit}, we complete the proof. 
 Let $d\leq \frac{1}{32\gamma}$ be the largest integer divisible by $4$. For sufficiently small $\gamma$, we have $d\geq \frac{1}{64\gamma}$ and $d$ is large enough for Lemma \ref{thm:dl-explicit}. For every $t\geq 1$, $G(d,t)$ is therefore a $\gamma$-edge-expander as $\frac{1}{32d}\geq \gamma$. Moreover, by Lemma \ref{lem:dl-regular-pro}, we have
$     \circum(G(d,t))\leq d2^t=\frac{n}{d+2}\leq 64\gamma n.$
 Letting $t\to\infty$ gives the required infinite family.
\end{proof}

\section{Concluding remarks}\label{sec:conc}
Our constructions show that even with regularity, both KS expanders and edge expanders may have circumferences bounded by an arbitrarily small constant fraction of their orders%, with this fraction depending on the expansion parameters
. The following proposition illustrates the differences among these notions.
\begin{proposition}\label{lem:expansion-relations}
Let $0<\varepsilon_1\le 1$ and $n$ be the order of expanders illustrated below. The following statements hold.
 \begin{enumerate}[label=\textup{(\roman*)}]
        \item\label{relation-1}
    Every $(\varepsilon_1,k)$-robust expander is a KS
    $(\varepsilon_1,k)$-expander. If the $(\varepsilon_1,k)$-robust expander is $d$-regular with $k\le d$, then it is a
    $\tfrac{\rho(n/2)}{2}$-edge-expander.
 \item\label{relation-2}
    A $d$-regular $\gamma$-edge-expander is a weak $(\varepsilon,c,\gamma d/2)$-robust expander whenever
    $0<\varepsilon/\log^c n\le\gamma/2$.
    Conversely, a $d$-regular weak $(\varepsilon,c,s)$-robust expander with $\varepsilon>0$ is also an $\tfrac{s}{d}$-edge-expander.
\item\label{relation-3}
    For any fixed $0<\varepsilon_1\le1\le s$ and $\varepsilon, c>0$,
    there exists a $d$-regular KS $(\varepsilon_1,d)$-expander which is not a weak $(\varepsilon,c,s)$-robust expander.
    Conversely, for fixed $\varepsilon_1>0$ and $0<\varepsilon, s\le1/70$, there exists a $d$-regular weak $(\varepsilon,2,s)$-robust expander which is not a KS $(\varepsilon_1,d)$-expander. 
\end{enumerate}
\end{proposition}

\begin{proof}
For \ref{relation-1}, taking $F=\varnothing$ in Definition~\ref{defn:HKL} gives a KS expander.
Let $G$ be an $n$-vertex $d$-regular graph with $k\le d$.
Let $X$ be a vertex set with $|X|\le 2n/3$, and $Y$ be the smaller one of
$X$ and its complement $\overline X$.
If $|Y|\ge k/2$, then
$e_G(Y,\overline Y)>d\rho(|Y|)|Y|$,
since otherwise deleting this cut contradicts the robust expander condition.
If $|Y|<k/2$, then
$e_G(Y,\overline Y)\ge |Y|(d-|Y|+1)\ge d|Y|/2$.
Since $\rho$ is decreasing on $[k/2,n/2]$ and
$\rho(k/2)\le 1/2$, both bounds are at least
$d\rho(n/2)|Y|$.
The result follows from $|Y|\ge |X|/2$.

For \ref{relation-2}, deleting at most $\gamma d|X|/2$ edges from a
 $\gamma$-edge-expander leaves at least $\gamma d|X|/2$
crossing edges.
As every vertex has degree at most $d$, this gives at least
$\gamma|X|/2$ external neighbors.
Conversely, weak robust expander condition implies
$e_G(X,\overline X)>s|X|$ for every $1\le |X|\le 2n/3$,
since deleting the whole cut gives a contradiction.

For \ref{relation-3}, we consider two graph constructions. First, let $G_1$ consist of two disjoint copies of $K_d$ joined by a perfect matching.
Then $G_1$ is $d$-regular on $n=2d$ vertices. For every $X\subseteq V(G_1)$ with $1\le |X|\le d$, we have $|N_{G_1}(X)|\ge |X|$.
As $\rho(x)\le\varepsilon_1\le 1$, $G_1$ is a KS $(\varepsilon_1,d)$-expander. However, taking $X$ to be one of the two $K_d$ and $F$ to be the perfect matching gives 
 $|F|=|X|$ and $N_{G_1-F}(X)=\varnothing$. Thus $G_1$ is not a weak $(\varepsilon,c,s)$-robust expander for $s\ge1$.

For the converse, consider the DL construction $G(d, t)$, where $d$ is sufficiently large and divisible by $4$, and $t=t(d)\ge 1$ is the smallest integer satisfying $ n=d2^t(d+2)\ge e^{\sqrt d}$.
Then %$\sqrt d\le\log n<\sqrt d+\log2$ and 
$d=(1+o(1))\log^2n$.
By Lemma~\ref{thm:dl-explicit}, $G(d, t)$ is a $d$-regular $1/(32d)$-edge-expander.
Part~\ref{relation-2} therefore guarantees that
$G(d, t)$ is a weak $(\varepsilon,2,s)$-robust expander for every $0<\varepsilon,s\le1/70$.
On the other hand, a lobe $X=V(Q_v)$ satisfies $|X|=d+1$ and $|N_G(X)|=1$.
This set lies in the KS testing range, but fails the KS condition as $\rho(d+1)(d+1)>1$, which is trivial when $d$ is  large.
\end{proof}

\subsection{Limitations of the constructions and further questions}

Our
construction in Section~\ref{sec:bs-construction} gives counterexamples at degree $\Theta(\log^2 n)$ and cannot support a substantially larger
degree. Moreover, its expansion constant depends on the prescribed circumference ratio. Indeed, let $X$ be the union of $\lfloor m/3\rfloor$ blocks.
Then $|X|=\Theta(n)$, whereas $N_G(X)\subseteq S$ and
$|S|=\Theta(n/d)$. Expansion at this scale forces
 $\tfrac nd=\Omega\!\left(\tfrac{n}{\log^2(n/d)}\right)$,
and hence 
 $d=O\!\left(\log^2\tfrac nd\right)$. The dependence of $\varepsilon_0$ on $\eta$ in Theorem \ref{thm:main} is
 unavoidable. For a whole
block $X=V(H_j)$ we have $|N_G(X)|=\beta d$ and $|X|=d+1$, so
expansion already forces
$ \beta d\ge
 \tfrac{\varepsilon_1(d+1)}
 {\log^2(15(d+1)/(\varepsilon_2d))}.$
On the other hand, the circumference estimate uses $\beta<\eta$.
Thus, within this construction, $\eta$ cannot tend to zero while $\varepsilon_1$ remains fixed.

This should be compared with the result of Letzter, Methuku and Sudakov~\cite{Nearly-H-cycle}, which finds nearly Hamilton cycles in regular weak robust expanders with a larger polylogarithmic degree. It is natural to ask whether the logarithmic-square scale is already sufficient to guarantee a long cycle in KS expanders. We posed the following question in the arxiv version $1$.

\begin{problem}\label{prom:H-pro}
For every $\varepsilon>0$, does there exist $C=C(\varepsilon)>0$ such that every sufficiently large $n$-vertex $d$-regular KS $(\varepsilon,d)$-expander with $d\ge C\log^2 n$ contains a Hamilton cycle?
\end{problem}
However, Li and Tang~\cite{Li-Tang-resolve prob} answered Problem~\ref{prom:H-pro} in the negative by constructing non-Hamiltonian graphs with degrees far above this scale. Since their graphs still contain nearly spanning cycles, the following weaker statement might still be true.

\begin{problem}\label{prob:ks-threshold}
For every $\varepsilon>0$ and $0<\delta<1$, does there exist
$C=C(\varepsilon,\delta)>0$ such that every sufficiently large
$n$-vertex $d$-regular KS $(\varepsilon,d)$-expander with
$d\ge C\log^2 n$ contains a cycle of length at least $\delta n$? %or even $(1-\delta)n$?
\end{problem}

The two constructions in Theorems~\ref{thm:main} and~\ref{thm:edge-main} fail the robust expander condition in Definition~\ref{defn:HKL}. Indeed, for fixed $\eps_1$ and $k=\Theta(d)$, a set of order $d$
in the robust expander must retain at least a constant multiple of
$d$ external neighbors after deleting up to $\Theta(d^2)$ edges.
In the BS construction, a whole block has only $\Theta(d)$ outgoing
edges, so deleting these edges violates the robust expander condition for
sufficiently large $d$. In the DL construction, a lobe has only one
external neighbor, so it already fails the corresponding KS
condition. 
Thus, the comparison in Proposition \ref{lem:expansion-relations}  leaves open whether a small circumference is
possible under the robust expander condition, even at the degree scale
of our first construction.

\begin{problem}\label{prob:hkl}
Do there exist fixed $\varepsilon_1,\varepsilon_2>0$ and infinitely
many pairs $(d,n)$ for which an $n$-vertex $d$-regular 
$(\varepsilon_1,\varepsilon_2d)$-robust expander in Definition~\ref{defn:HKL} has circumference
$o(n)$? In particular, is such a construction possible with
$d=\Theta(\log^2n)$?
\end{problem}

\medskip
\noindent
\textbf{Acknowledgment.}
We thank Zach Hunter for pointing out, after the first version of this paper was posted, that the Hamiltonicity assertion in Conjecture~\ref{conj} can be disproved by Meredith~\cite{Meredith}.
For every integer $r=3,4$ or $r\geq 8$, the construction of Meredith gives an $r$-regular, $r$-connected non-Hamiltonian graph on $20r-10$ vertices, satisfying the KS condition for a fixed sufficiently small expansion constant.
Notice that these graphs still contain cycles of length $(1-o(1))n$, whereas ours can achieve circumferences with arbitrarily small portion of $n$.
In fact, we initially tried to prove that $n$-vertex regular KS expanders with sufficiently large constant degree contain a cycle of length $0.51n$, as this would resolve a related question asking for the existence of a cycle with linearly many chords. Unable to do so, we eventually realized that it is not true, arriving at Theorem~\ref{thm:main}.

We also thank Oliver Janzer for pointing out the subtle distinction between the KS expanders in Definition~\ref{def:KS} and the weak robust expanders in Definition~\ref{defweakro}. His remark motivated us to consider the second construction in Section~\ref{sec:dl-construction} satisfying weak robust expander properties.

\medskip
\noindent
\textbf{Declaration of AI use.}
During the preparation of this manuscript, the authors used ChatGPT 5.6 as an auxiliary tool, primarily to check technical details in the verification of the expansion properties of the constructions and to improve the language of the manuscript.

\begingroup
    \footnotesize

\end{document}